\documentclass[a4paper, 11pt]{amsart} 
\usepackage{geometry}
\usepackage{fullpage}
\usepackage{amssymb,amsmath,amsthm,ascmac,array,bm}
\usepackage{graphicx}
\usepackage{color}
\usepackage{enumerate}
\usepackage{mathrsfs} 
\usepackage{url}
\usepackage{longtable}

\usepackage{tikz}
\usetikzlibrary{intersections,calc,arrows.meta}
\usetikzlibrary{quotes}
\usetikzlibrary{knots}

\usepackage{comment}
\usepackage{hyperref}
\usepackage[all]{xy}

\usepackage{xcolor}
\usepackage[capitalize,nameinlink,noabbrev,nosort]{cleveref}
\newtheorem{theoremcounter}{Theorem Counter}[section]

\theoremstyle{definition}
\newtheorem{defn}[theoremcounter]{Definition}

\theoremstyle{plain}
\newtheorem{lem}[theoremcounter]{Lemma}
\newtheorem{prop}[theoremcounter]{Proposition}

\newtheorem{thm}[theoremcounter]{Theorem}

\numberwithin{equation}{section}

\allowdisplaybreaks

\makeatletter
\@namedef{subjclassname@2020}{
\textup{2020} Mathematics Subject Classification}
\makeatother 



\begin{document}

   \title{On genus theory for $3$-manifolds\\ in arithmetic topology}
   \author{Hirotaka TASHIRO}
\maketitle
 
\begin{abstract}
Based on the analogies of arithmetic topology, we show a topological analogue of Hilbert's Satz 90 for id\`ele groups and utilize our previously established Hasse norm principle
to present a proof of an Iyanaga--Tamagawa type genus formula for finite abelian branched covers over integral homology 3-spheres in a very parallel manner to the case of number theory.
\end{abstract}

\tableofcontents
\section{Introduction}
\subsection{Backgrounds}
After Artin, Takagi, and Chevalley founded the id\`ele class field theory for Abelian extensions of number fields in the first half of the 20th century, many applications of general class field theory were found in algebraic number theory. Among others, Iyanaga--Tamagawa \cite{IyanagaTamagawa} addressed this by establishing a certain genus formula for cyclic extensions of the rationals, utilizing the Hasse norm principle and a version of Hilbert's Satz 90 for id\`ele groups.

Meanwhile, drawing the analogy between knots and primes, or 3-manifolds and the ring of integers of number fields, analogues of id\`elic class field theory and the Hasse norm principle have been developed by Niibo  \cite{Niibo1}, Ueki \cite{NiiboUeki, NiiboUeki2023RMS}, and the author \cite{Tashiro2024}. The purpose of this paper is to present a topological analogue of Iyanaga--Tamagawa's genus formula for cyclic branched covers over an integral homology 3-sphere, in a manner that \emph{closely parallels} the original proof for number fields. We note that a similar formula was first formulated by Morishita over $S^3$ in \cite{Morishita2001g}, though his proof differs from ours; our approach relies on id\`elic class field theory, our previous work on the Hasse norm principle, and a newly established version of Hilbert's Satz 90 for id\`ele groups of 3-manifolds endowed with very admissible links.

We also remark that several other versions of the genus formula for number fields exist; analogues of Yokoi \cite{Yokoi1967} and Furuta \cite{Furuta1967} have been developed by Ueki \cite{Ueki1, Ueki3} tailored to their respective objectives. 

Now, let us recall some basic dictionary of dictionary in arithmetic topology, which we will use in this paper. Let $F$ be a number field and let $\mathcal{O}_F$ denote the ring of integers of $F$. For any a 3-manifold $M$, we put a very admissible link $\mathcal{L}$ (cf. Section 2).\\

\begin{center}
\begin{longtable}{|c|c|}
\hline

compactified spectrum of $\mathcal{O}_F$ & oriented connected closed 3-manifold \\

$\overline{{\rm Spec}(\mathcal{O}_{F})}$ & $M$\\

\hline

maximal ideal of $\mathcal{O}_F$ &knot in $M$  \\

${\rm Spec}(\mathcal{O}_F/{\mathfrak{p}})\hookrightarrow {\rm Spec}(\mathcal{O}_{F})$ &$K:S^{1}\hookrightarrow M$\\

\hline

finite set of maximal ideals of $\mathcal{O}_F$ & link in $M$ \\

$S=\{\mathfrak{p}_{1}, \cdots \mathfrak{p}_{n}\}$ &$L=K_{1}\sqcup \cdots \sqcup K_{n}$\\

\hline
set of all primes of $F$ & very admissible link in $M$ \\

$\mathcal{P}$ &$\mathcal{L}$\\

\hline

multiplicative group of $F$& group of $2$-chains of $M$ \\

$F^{\times}$ &$C_2(M; \mathbb{Z})$ \\

\hline

group of fractional ideals of ${\mathcal O}_{F}$ & group of $1$-cycles of $M$ \\

$J_F$ & $Z_1(M)$ \\

\hline

boundary map & boundary map \\

$\partial _{F}:F^{\times} \rightarrow J_{F}$ &$\partial _{M}:C_{2}(M) \rightarrow  Z_{1}(M)$ \\

$a \mapsto (a)$ & $\Sigma \mapsto \partial _M \Sigma$\\

\hline

(narrow) ideal class group of $F$ &1st homology group of $M$ \\

$H(F) = {\rm Coker}(\partial _{F})\ (H^+(F))$ & $H_{1}(M) = {\rm Coker}(\partial _{M})$\\

\hline

id\`ele group of $F$ & id\`ele group of $(M,\mathcal{L})$\\

$I_F$ & $I_{M,\mathcal{L}}$\\

\hline

principal id\`ele group of $F$ & principal id\`ele group of $(M,\mathcal{L})$\\

$P_F$ & $P_{M,\mathcal{L}}$\\

\hline

id\`ele class group of $F$ & id\`ele class group of $(M,\mathcal{L})$\\

$C_F=I_F/P_F$ & $C_{M,\mathcal{L}}=I_{M,\mathcal{L}}/P_{M,\mathcal{L}}$\\

\hline

 finite (ramified) extension  &finite (branched) covering\\

$E/F$ &$h:N\rightarrow M$\\

\hline

$n$-th power residue symbol  &mod $n$ linking number  \\
of $a\in F^\times$ and a prime $p\in\mathbb{N}$&of a link $L$ and a knot $K$\\
$(\frac{a}{p})_n$ &${\rm lk}(L,K)\ {\rm mod}\ n$ \\

\hline

\end{longtable}
\end{center}

\subsection{Genus theory in number theory}
Next, we overview Iyanaga--Tamagawa's results in Number theory \cite{IyanagaTamagawa}; 
a version of Hilbert's Satz 90 for id\`ele groups and a genus formula for finite cyclic extensions. 
\begin{thm}[Hilbert's Satz 90 for id\`ele groups {\cite[Chapter 7, Corollary 7.4]{CasselsFrohlich}}] \label{Hil90.nt} 
Let $E/F$ be a finite cyclic extension of number fields with the Galois group ${\rm Gal}(E/F)=\langle\tau \rangle$, and let ${\rm N}_{E/F}:I_E\to I_F$ denote the norm map of id\`ele groups. Then, $H^1({\rm Gal}(E/F),I_E)=0$ holds. Namely, for any $a \in I_E$ with ${\rm N}_{E/F}(a)=1$, there exists $b \in I_E$ such that $a=b^{\tau-1}$. 
\end{thm}
Now let $E/F=k/\mathbb{Q}$ with degree $n$ unramified outside $p_1,\ldots,p_r,\infty$. Here, suppose that $p_1,\dots ,p_r$ are prime numbers with $p_i\equiv1$ mod $n$ for each $i$. Let $e_i$ denote the ramification index of $p_i$ and put $\zeta_m:={\rm exp}(\frac{2\pi\sqrt{-1}}{m})\in \mathbb{C}$ for any $m\in\mathbb{Z}_{>0}$, so we have $\langle \zeta_{e_i}\rangle\cong \mathbb{Z}/e_i\mathbb{Z}$. Fix an embedding $\mathbb{Q}(\langle \zeta_{e_i}\rangle)\hookrightarrow\mathbb{Q}_{p_i}$ into the $p_i$-adic number field, so the $n$-th power residue symbol $(\frac{\mathfrak{a}}{p_i})_n$ of $\mathfrak{a}\in J_k$ over $p_i$ is defined to be a value in $\langle \zeta_{e_i}\rangle$. An ideal representing an ideal class of the narrow ideal class group $H^+(k)$ will be taken to be that of $\mathcal{O}_k$ disjoint from the ramified locus $S=\{p_1,\dots,p_r,\infty\}$. 
\begin{defn}\label{def.genus.nt}
We say $[\mathfrak{a}],[\mathfrak{b}]\in H^+(k)$ belong to \emph{the same genus} and write as $[\mathfrak{a}]\approx[\mathfrak{b}]$ if the following holds:
\[
\Big(\frac{{\rm N}\mathfrak{a}}{p_i}\Big)_{e_i}=\Big(\frac{{\rm N}\mathfrak{b}}{p_i}\Big)_{e_i} \mbox{for every }i.
\] 
This definition does not depend on the choice of ideals $\mathfrak{a},\mathfrak{b}$ representing the classes $[\mathfrak{a}],[\mathfrak{b}]$. We define \emph{the genus number} of $k/\mathbb{Q}$ by $g_{k/\mathbb{Q}}=\#(H^+(k)/H^+(k)^{\tau-1})$. 
\end{defn}
Iyanaga--Tamagawa's results may be explicitly described as follows (\cite[Theorem4,5]{IyanagaTamagawa}, see also \cite[Subsection 6.3]{Morishita2024}, \cite{Morishita2001g}).
\begin{thm}[Genus theory {\cite[Theorem 4,5]{IyanagaTamagawa}}] \label{thm.genus.nt} 
Let $k/\mathbb{Q}$ be as in the above and consider the group homomorphism $\chi:H^+(k)\rightarrow \prod_{i=1}^r\langle\zeta_{e_i}\rangle; [\mathfrak{a}]\mapsto ((\frac{{\rm N}\mathfrak{a}}{p_i})_{e_i})$. Then,
\[{\rm Im}\chi=\{(x_i)\in \prod_{i=1}^r\langle\zeta_{e_i}\rangle\mid\prod_{i=1}^r x_i^{\frac{n}{e_i}}=1\in\langle\zeta_n\rangle\},\  {\rm Ker}\chi=H^+(k)^{\tau-1},
\]
 and
 \[H^+(k)/\approx\ \  \simeq H^+(k)/H^+(k)^{\tau-1} \simeq {\rm Im}\chi.
\]
In particular, $$g_{k/\mathbb{Q}}=\frac{\prod_i e_i}{n}.$$
\end{thm}
\cref{thm.genus.nt} is proved by using Theorem 1.1 and Hasse norm theorem for $k/\mathbb{Q}$. For this details of this argument, we refer to \cite[Section 2]{IyanagaTamagawa}, \cite[Theorem 6.3]{Morishita2024}.
\subsection{Genus theory in 3-dimensional topology and main results}
Now, let us turn to a topological analogue for 3-manifolds. Let $M$ be an oriented, connected, closed $3$-manifold. We fix a very admissible link $\mathcal{L}$ in $M$, which is a link consisting of countably many (finite or infinite) tame components with certain conditions and plays an analogous role to the set of all primes, and we shall employ the notions of the id\`{e}le group $I_{M,\mathcal{L}}$ and the principal id\`{e}le group $P_{M,\mathcal{L}}$ introduced in \cite{NiiboUeki}. Note that $P_{M,\mathcal{L}}$ is defined as the image of the diagonal map $\Delta _{M,\mathcal{L}}:H_2 (M,\mathcal{L})\to I_{M,\mathcal{L}}$ (cf. Section 2). For a finite cover $f: N\rightarrow M$ branched over a finite sublink $L$ of $\mathcal{L}$, $f^{-1}(\mathcal{L})$ is again a very admissible link of $N$. Note that the induced map $f_* : I_{N,f^{-1}(\mathcal{L})} \rightarrow I_{M,\mathcal{L}}$ is an analogue of the norm map. One of our main theorem, which may be regarded as a topological analogue for 3-manifolds of \cref{Hil90.nt}, is stated as follows.

\begin{thm}[Theorem 3.1]
Let $M$ be an oriented connected closed 3-manifolds endowed with a very admissible link $\mathcal{L}$. Let $f:N\to M$ be a finite cyclic covering branched over a finite sublink $L_0$ of $\mathcal{L}$ with the Galois group ${\rm Gal}(f)=\langle \tau \rangle$. Then, $\widehat{H}^1({\rm Gal}(f),I_{N,f^{-1}(\mathcal{L})})=0$ holds. Namely, for any $a \in I_{N,f^{-1}(\mathcal{L})}$ with $f_*(a)=0$, there exists $b \in I_{N,f^{-1}(\mathcal{L})}$ such that $a=(\tau-1)b$.
\end{thm}

Suppose in addition that $M$ is an integral homology 3-sphere ($\mathbb{Z}$HS$^3$), namely, $H_*(M)=H_*(S^3)$ holds. Put deg$f:= n$ and a link $L_0=K_1\sqcup\cdots\sqcup K_r$ of $M$. Let $e_i$ denote the branch index of each $K_i$ in $f$. A 1-cycle representing a homology class of $H_1(N)$ will be taken to be disjoint from $f^{-1}(L_0)$. 
\begin{defn}\label{def.genus.ldt}
We say that $[z],[w]\in H_1(N)$ belong to \emph{the same genus} and write as $[z]\approx[w]$ if the following holds:
$$
{\rm lk}(f_*(z),K_i)\equiv{\rm lk}(f_*(w),K_i) \ {\rm mod}\ e_{i} \mbox{ for every }i.
$$
This notion does not depend on the choice of representatives $z,w$ of classes. We define \emph{the genus number} of $f:N\to M$ by $g_f=\#(H_1(N)/(\tau-1)H_1(N))$.
\end{defn}

Another main result is the following analogue of \cref{thm.genus.nt}. 
\begin{thm}[Theorem 4.1]
Notations and assumptions being as above, consider the group homomorphism $\chi:H_1(N)\to \prod_{i=1}^r\mathbb{Z}/e_i\mathbb{Z}; [z]\mapsto ({\rm lk}(f_*(z),K_i)\mbox{ mod}\ e_i)_i$. Then,
\[
{\rm Im}\chi=\{(x_i)_i\in\prod_{i=1}^r\mathbb{Z}/e_i\mathbb{Z}\mid\sum_{i=1}^r \frac{n}{e_i}x_i=0\in\mathbb{Z}/n\mathbb{Z}\},\ {\rm Ker}\chi=(\tau-1)H_1(N),
\]
 and
 \[H_1(N)/\approx\ \  \simeq H_1(N)/(\tau-1)H_1(N) \simeq {\rm Im}\chi.
\]
In particular, $$g_f=\frac{\prod_i e_i}{n}.$$
\end{thm}
This paper is organized as follows. In Section 2, we review the id\`ele group for a $3$-manifold $M$ endowed with a very admissible link $\mathcal{L}$ and the some theorems. In Section 3, we prove a new result: a topological analogue of Hilbert's Satz 90 for the id\`ele group of $M$ in two ways. In Section 4, we prove a topological analogue of the genus theory by the id\`ele group of $3$-manifolds.\\

\noindent 
\textbf{Notation and convention}. 3-manifolds are assumed to be oriented, connected, and closed. A knot is assumed to be tame. We denote by $V_K$ a tubular neighborhood of a knot $K$ and put $V_L:=\sqcup_{K\subset L}V_K$. For a manifold $M$ and its submanifold $A$, we denote by $H_{n}(M)$ and $H_{n}(M,A)$ the $n$-th homology group and the $n$-th relative homology group with coefficients in $\mathbb{Z}$. 
The branch set of a branched covering of a 3-manifold is assumed to be a finite link. 
When $f:N\rightarrow M$ is a branched Galois covering, ${\rm Gal}(f)$ denotes the Galois group of $f$. For a knot $K$ in a 3-manifold $M$, $\mu _{K}$ and $\lambda _{K}$ denotes the meridian and the (chosen) longitude of $K$ regarded as elements in $H_1(\partial V_K)$ and several other groups. When $M$ is an integral homology 3-sphere, $\lambda _{K}$ denotes the preferred longitude of $K$. When $K, K'$ are disjoint knots in an integral homology 3-sphere $M$, then ${\rm lk}(K,K')$ denotes their linking number. For each component $K$ of a link $\mathcal{L}$, the orientation of the longitude $\lambda_K \subset \partial V_K$ is the same as $K$ and the orientation of the meridian $\mu_K \subset \partial V_K$ is determined by ${\rm lk}(\mu_K,K) = 1$.

\section{Very admissible link and id\`eles} 
\label{s2} 
In this section, we recollect the notion of a very admissible link $\mathcal{L}$ in a 3-manifold $M$, the id\`ele  group and the principal id\`{e}le group of $(M,\mathcal{L})$. Also, we introduce some properties which are proved by the global reciprocity law for 3-manifolds, which is the important fact used in our main result. We consult \cite{Morishita2024}, \cite{NiiboUeki}, \cite{Tashiro2024} as basic references for this section. For an arithmetic counterpart, we refer to \cite{Neukirch}.

First, we recall the Hilbert theory for a branched cover. Let $M$ be an oriented, connected, closed 3-manifold and $L_0=K_1\sqcup \cdots \sqcup K_r\ (r\in \mathbb{Z}_{\ge1})$ a finite link in $M$. Let $f:N\to M$ be a finite Galois cover of degree $n$ branched over $L_0$ with $G:={\rm Gal}(f)$. Let $K$ be a knot in $M$ with $K\subset L_0$ or  $K\cap L_0=\emptyset$ and put $f^{-1}(K)=J_1\sqcup\cdots\sqcup J_{c_K}$, where $c_K$ denotes the number of components of $f^{-1}(K)$. 
\begin{defn}
We define \emph{the decomposition group} of $J_i\subset f^{-1}(K)$ in $f$ by $D_{J_i}:=\{g\in G\mid g(J_i)=J_i\}$,  so that we have a bijection $G/D_{J_i}\to \{J_1,\cdots,J_{c_K}\}$. Note that there is a natural surjective homomorphism $D_{J_i}\to {\rm Gal}(J_i/K); g\mapsto g|_{J_i}$ and define \emph{the inertia group} $I_{J_i}$ of $J_i$ in $f$ by its kernel. 
\end{defn}
Thus, we obtain the following exact sequence $$1\to I_{J_i}\to D_{J_i}\to {\rm Gal}(J_i/K)\to 1.$$ The orders of $I_{J_i}$ and ${\rm Gal}(J_i/K)$ coincide with \emph{the branched index} $e_K$ of $K$ and the covering degree of $J_i\to K$, respectively. 
\begin{lem}[Hilbert theory for branched covering (cf.{\cite[Theorem 5.1.1]{Morishita2024}},{\cite[Section 2]{Ueki1}})]
\label{HilTh} Let $N\to T_{J_i}\to Z_{J_i}\to M$ denote the decomposition of $f:N\to M$ corresponding to the sequence of subgroups $e\leq I_{J_i}\leq D_{J_i}\leq {\rm Gal}(f)$. Then, every knot over $K$ completely branches in $N\to T_{J_i}$, is completely inert (namely, does not branch nor decompose) in $T_{J_i}\to Z_{J_i}$, and $K$ completely decomposes into a $c_K$-component link in $Z_{J_i}\to M$. We have $e_K={\rm deg}(N\to T_{J_i})$, $d_K={\rm deg}(T_{J_i}\to Z_{J_i})$, $c_K={\rm deg}(Z_{J_i}\to M)$, and $n=c_Kd_Ke_K$. 
In particular, if ${\rm Gal}(f)$ is abelian, then these subgroups and subcovers are independent of the choice of $J_i$.
\end{lem}

Now, we recall the notion of a very admissible link $\mathcal{L}$ of 3-manifold $M$ and the id\`ele and principal id\`ele groups of $M$.
\begin{defn}
Let $M$ be an oriented connected closed 3-manifold and let $\mathcal{L}$ be a link of $M$ consisting of countably many (finite or infinite) tame components. We call  $\mathcal{L}$  a {\em very admissible link} if $\mathcal{L}$ satisfies the following condition: For any finite cover $f:N\rightarrow M$ branched over a finite sublink $L_0$ of $\mathcal{L}$, $H_{1}(N)$ is generated by the homology classes of components of $f^{-1}(\mathcal{L})$.
\end{defn}
By Definition 1.1, if $\mathcal{L}$ is a very admissible link of $M$ and $f:N \rightarrow M$ is a finite covering branched over a finite sublink $L_0$ of $\mathcal{L}$, then $f^{-1}(\mathcal{L})$ is again a very admissible link of $N$. The following theorem is fundamental.

\begin{prop}[{\cite[Theorem 2.3]{NiiboUeki}}]\label{exVAL} 
Let $M$ be an oriented, connected, closed 3-manifold and $L$ a link of $M$ consisting countably many tame components. Then, there exists a very admissible link $\mathcal{L}\subset M$ containing $L$. In particular, for any $M$, there exists a very admissible link $\mathcal{L}$.
\end{prop}
Hereafter, we fix a very admissible link $\mathcal{L}$ in a 3-manifold $M$ and state id\`ele theory for a pair $(M,\mathcal{L})$. We may assume that $\mathcal{L}$ is endowed with tubular neighborhoods. Indeed, by the method of blow-up, we may essentially assume that $\mathcal{L}$ is endowed with a tubular neighborhood $V_\mathcal{L}=\sqcup_{K\subset \mathcal{L}} V_K$, which is the disjoint union of tori (cf.~\cite{NiiboUeki2023RMS}). Note that we may instead consider families of tubular neighborhoods with natural identifications of their groups (cf.~\cite{NiiboUeki}), or work over the system of formal tubular neighborhoods as well (cf.~\cite{Mihara2019Canada}).

\begin{defn} We define the {\em id\`ele group} of $(M,\mathcal{L})$ by
$$I_{M,\mathcal{L}}:=\{(a_{K})_{K}\in \underset{K\subset \mathcal{L}}{\prod}H_{1}(\partial V_{K})\mid a_K\in \mathbb{Z}[\mu_K]\mbox{ for almost all} \  K \subset \mathcal{L}\},$$
where $K \subset \mathcal{L}$ runs through all components of $\mathcal{L}$. An element of $I_{M,\mathcal{L}}$ is called an {\em id\`{e}le} of $(M,\mathcal{L})$.
\end{defn}

In order to define the principal id\`ele group, we define the $H_{2}(M,\mathcal{L})$ and construct the $\Delta _{M,\mathcal{L}}:H_2(M,\mathcal{L})\to I_{M,\mathcal{L}}$. For any finite sublinks $L,L'$ of $ \mathcal{L}$ with $L\subset L'$, there exists a natural injection $j_{L,L'}:H_{2}(M,L)\hookrightarrow H_{2}(M,L')$. Then $((H_2(M,L))_{L \subset \mathcal{L}}, (j_{L,L'})_{L \subset L'\subset \mathcal{L}})$ forms a direct system.

\begin{defn}
We define $H_{2}(M,\mathcal{L})$ by the following direct limit:
$$
H_{2}(M,\mathcal{L}):=\underset{L\subset \mathcal{L}}{\varinjlim}H_{2}(M,L)
=\underset{L\subset \mathcal{L}}{\bigsqcup}H_{2}(M,L)/\thicksim.\\
$$
Here, for $S_{L}\in H_{2}(M,L)$ and $S_{L'}\in H_{2}(M,L')$, we write $S_{L}\thicksim S_{L'}$ if and only if there exists a finite link $ L''\subset \mathcal{L}$ such that $L\cup L'\subset L'' \ \mbox{and}\   j_{L,L''}(S_{L})=j_{L',L''}(S_{L'})$.
\end{defn}
Let $L$ be a finite sublink of $\mathcal{L}$. We put $V_{L}:=\bigsqcup _{K\subset L}V_{K}$ and $X_{L}:=M\setminus {\rm Int}(V_L)$ $(\simeq M\setminus L)$. By the excision isomorphism ${\rm exc}$ and the relative homology exact sequence for a pair $(M,V_{L})$, we obtain a sequence
$$
H_{2}(M,L)\overset{\cong}{\rightarrow} H_{2}(M,V_{L})\overset{{\rm exc}}{\rightarrow} H_{2}(X_{L},\partial V_{L})\overset{\partial}{\rightarrow} H_{1}(\partial V_{L}).
\leqno{(1.6)}
$$
Moreover, for any finite links $L,L'\subset \mathcal{L}$ with $L\subset L'$, we have a commutative diagram
$$
\begin{array}{ccc}

H_{2}(M,L') & \overset{\partial}{\rightarrow} & H_{1}(\partial V_{L'})\\

j_{L,L'}\uparrow &\circlearrowright & \downarrow p_{L',L} \\

H_{2}(M,L) & \overset{\partial}{\rightarrow} & H_{1}(\partial V_{L}).

\end{array}
\leqno{(1.7)}
$$
Taking the direct limit for $H_2(M,L)$ and the projective limit for $H_1(V_L)$ with respect to finite sublinks $L \subset \mathcal{L}$, we obtain a natural homomorphism $\Delta_{M,\mathcal{L}}:H_{2}(M,\mathcal{L})\rightarrow I_{M,\mathcal{L}}$ called the \emph{diagonal map}. 
\setcounter{theoremcounter}{7}
\begin{defn}
We define the {\em principal id\`ele group} of $(M,\mathcal{L})$ by
$$
P_{M,\mathcal{L}}:={\rm Im}(\Delta_{M,\mathcal{L}}).
$$
An element of $P_{M,\mathcal{L}}$ is called a principal id\`{e}le of $(M,\mathcal{L})$.
Also, we define the {\em id\`ele class group} of $(M,\mathcal{L})$ by
$$
C_{M,\mathcal{L}}:=I_{M,\mathcal{L}}/P_{M,\mathcal{L}}
$$
An element of $C_{M,\mathcal{L}}$ is called an id\`{e}le class of $(M,\mathcal{L})$.
\end{defn}
Now, we introduce some properties on the id\`ele group of 3-manifolds. The following theorems will be used in Section 4.
\begin{lem}[{\cite[Theorem 5.4]{NiiboUeki}}]\label{GCFT}
For any finite abelian covering $f:N\to M$ branched over a finite sublink $L_0$ of $\mathcal{L}$, we have an isomorphism 
$$
C_{M,\mathcal{L}}/h_*(C_{N,f^{-1}\mathcal{L}})\cong {\rm Gal}(f).
$$
\end{lem}
For a finite sublink $L\subset\mathcal{L}$, we define the subgroup $U_{\mathcal{L}\setminus L}$ of $I_{M,\mathcal{L}}$ by $U_{\mathcal{L}\setminus L} :=\prod _{K\subset \mathcal{L} \setminus L}\mathbb{Z}[m_{K}]$. Especially, if $L=\emptyset$, then we call $U_{M,\mathcal{L}}:=U_{\mathcal{L}\setminus\emptyset}$ the {\em unit id\`ele group} of $(M,\mathcal{L})$.
\begin{lem}[{\cite[Lemma 5.7]{NiiboUeki}}] \label{IdeleHom}
We have an isomorphism
$$
I_{M,\mathcal{L}}/(P_{M,\mathcal{L}}+U_{\mathcal{L}\setminus L})\cong H_{1}(X_{L}).
$$
In particular, if $L=\emptyset$, then $I_{M,\mathcal{L}}/(P_{M,\mathcal{L}}+U_{M,\mathcal{L}})\cong H_1(M)$ holds. Moreover, if $M$ is an integral homology 3-sphere, then $I_{M,\mathcal{L}}=P_{M,\mathcal{L}}\oplus U_{M,\mathcal{L}}$. 
\end{lem}
\begin{lem}[{\cite[Theorem 3.1]{Tashiro2024}}]\label{HNP}
Let $M$ be an integral homology 3-sphere. For any finite cyclic covering $f:N\to M$ branched over a finite sublink $L_0$ of $\mathcal{L}$, we have the following formula
$$
f_*(I_{N,f^{-1}(\mathcal{L})})\cap P_{M,\mathcal{L}}=f_*(P_{N,f^{-1}(\mathcal{L})}).
$$
\end{lem}
\begin{lem}[{\cite[Proposition 2.6]{Tashiro2024}}]\label{diagonal}
Let $M$ be an integral homology 3-sphere endowed with a very admissible link $\mathcal{L}$. Take $[A]\in H_2(M,\mathcal{L})$. Then, there is a finite sublink $L\subset\mathcal{L}$ such that $A\in H_2(M,L)$ and we can write $A=\sum_{K\subset L}c_K[S_K]$ with $c_K\in\mathbb{Z}$. Let $\Delta_{M,\mathcal{L}}=(a_K)_K\in I_{M,\mathcal{L}}$. Then, we have the following formula: 
$$
a_{K}=\begin{cases}
\displaystyle c_{K}[\lambda _{K}]- (\sum_{K'\subset L\setminus K} {\rm lk}(K,K')c_{K'})[\mu _{K}] & (K\subset L) \\
\displaystyle -\sum_{K'\subset L} {\rm lk}(K,K')c_{K'} [\mu _{K}] &(K\subset \mathcal{L}\setminus L).
\end{cases}
$$
\end{lem}

\section{Hilbert's Satz 90 for id\`eles} \label{s3}
In this section, we prove topological Hilbert's Satz 90 for id\`ele groups which is the new result in two ways. Our first proof uses the Hilbert theory for branched covers and visualizes the situation. Our second proof uses Tate  cohomology and is immediate.
\begin{thm}\label{Hil90.ldt}
Let $M$ be an oriented connected closed 3-manifold endowed with a very admissible link $\mathcal{L}$. 
Let $f:N\to M$ be a cyclic branched cover branched over a finite sublink $L_0$ of $\mathcal{L}$ with the Galois group $G:={\rm Gal}(f)=\langle \tau \rangle$. Put $\widetilde{\mathcal{L}}:=f^{-1}(\mathcal{L})$ and $f$ induces the homomorphism $f_*:I_{N,\widetilde{\mathcal{L}}}\to I_{M,\mathcal{L}}$.
Then, $\widehat{H}^1(G,I_{N,\widetilde{\mathcal{L}}})=0$ holds. Namely, for any $a\in I_{N,\widetilde{\mathcal{L}}}$ with $f_*(a)=0$, there exists $b\in I_{N,\widetilde{\mathcal{L}}}$ such that $a=(\tau-1)b$.
\end{thm}
\begin{proof}
First, we present $f_*$ explicitly by using \cref{HilTh}. Let $a=(a_J)_J\in I_{N,\widetilde{\mathcal{L}}}$. For each component $J\subset \widetilde{\mathcal L}$, put $K=f(J)$ and $f^{-1}(K):=J_1\sqcup\cdots\sqcup J_{c_K}$ is a $c_K$-component link including $J$. Since $I_{N,\widetilde{\mathcal{L}}}\subset \prod_{J \subset \widetilde{\mathcal{L}}}H_1(\partial V_J)$, we can write $a_J=l_J[\lambda_J]+m_J[\mu_J]$. Simply, $l_i:=l_{J_i}$, $m_i:=m_{J_i}$. By \cref{HilTh}, the restriction map $f|_{\partial V_{J_i}}$ induces ${f_{\partial V_{J_i}}}_*:H_1(\partial V_{J_i})\to H_1(\partial V_K); l_i[\lambda_{J_i}]+m_i[\mu_{J_i}]\mapsto d_Kl_i[\lambda_K]+e_Km_i[\mu_K]$. Thus, we have $$f_*((a_{J})_J)=(d_K\sum_{J_i\subset f^{-1}(K)}l_i[\lambda_K]+e_K\sum_{J_i\subset f^{-1}(K)}m_i[\mu_K])_K.$$

Assume $f_*(a)=0$, i.e., $\sum_{J_i\subset f^{-1}(K)}l_i=\sum_{J_i\subset f^{-1}(K)}m_i=0$. Then, we construct $b=(b_J)_J\in I_{N,\widetilde{\mathcal{L}}}$ such that $a=(\tau-1)b$. Put $\sigma_K:=\tau^{d_Ke_K} \mbox{ and } \langle\sigma_K\rangle\cong d_Ke_K\mathbb{Z}/n\mathbb{Z}\cong\mathbb{Z}/c_K\mathbb{Z}$. We may assume $\sigma_K(J_i)=J_{i-1}$ for every $i\in\mathbb{Z}/c_K\mathbb{Z}$. Define $(b'_J)_J\in I_{N,\widetilde{\mathcal{L}}}$ such that $\bigoplus_{J_i\subset f^{-1}(K)}H_1(\partial V_{J_i})$-component $(b'_{J_i})_i$ over $K$ is denoted by $(x_i[\lambda_{J_i}]+y_i[\mu_{J_i}])_i$ and $x_1=0,x_2=l_1,\ldots, x_{c_K}=l_1+\cdots+l_{c_K-1}, y_1=0,y_2=m_1,\ldots, y_{c_K}=m_1+\cdots+m_{c_K-1}$. Then, we have $(\sigma_K-1)(b'_{J_i})_i=((x_{i+1}-x_i)[\lambda_i]+(y_{i+1}-y_i)[\mu_i])_i=(l_i[\lambda_i]+m_i[\mu_i])_i=(a_{J_i})_i.$ Note $(\sigma_K-1)=(\tau-1)(\tau^{d_Ke_K-1}+\tau^{d_Ke_K-2}+\cdots+\tau+1)$ and when $(l_i)_i\neq0$, then $(x_i)_i\neq0$. If we define $b:=(b_J)_J\in I_{N,\widetilde{\mathcal{L}}}$ such that $\oplus_{J_i}H_1(\partial V_{J_i})$-component $(b_{J_i})_i$ is $(\tau^{d_Ke_K-1}+\cdots+1)(b'_{J_i})_i$, then $a=(\tau-1)b$. 
\end{proof}
\begin{proof}[Alternative proof]
When we put $\mathcal{N}:I_{N,\widetilde{\mathcal{L}}}\to I_{N,\widetilde{\mathcal{L}}}; a\mapsto \sum_{i=1}^n\tau^ia$, $f_*(a)=0$ if and only if $\mathcal{N}(a)=0$. Take any $a\in I_{N,\widetilde{\mathcal{L}}}$ with $f_*(a)=0$, i.e., $\sum_{i=1}^n\tau^ia=0$. For any free $\mathbb{Z}G$-module $\mathcal{M}$, the Tate cohomology $\widehat{H}^i(G,\mathcal{M})=0$. Since $I_{N,\widetilde{\mathcal{L}}}$ is a free $\mathbb{Z}/n\mathbb{Z}$-module, $\widehat{H}^1(G, I_{N,\widetilde{\mathcal{L}}})=0$. Thus, there exists $b\in I_{N,\widetilde{\mathcal{L}}}$ such that $a=(\tau-1)b$. \end{proof}

\section{Genus theory for $\mathbb{Z}$HS$^{3}$} \label{s4}
In this section, we present the proof of the topological genus theory that is perfectly parallel to the case of number theory. Hereafter, we assume $M$ is a $\mathbb{Z}{\rm HS}^3$. Let $L_0=K_1\sqcup \cdots \sqcup K_r$ be a $r(\in\mathbb{Z}_{>0})$-components link of $M$. Let $f:N\to M$ be a cyclic covering of degree $n$ branched over $L_0$ with the Galois group Gal$(f)=\langle\tau\rangle$ and let $e_i$ denote the branch index of $K_i$. A 1-cycle representing a homology class of $H_1(N)$ will be taken to be disjoint from $f^{-1}(L_0)$. We say that $[z],[w]\in H_1(N)$ belong to the same genus and write as $[z]\approx[w]$, if the following holds:
\[
{\rm lk}(f_*(z),K_i)\equiv{\rm lk}(f_*(w),K_i) \ {\rm mod}\ e_i \mbox{ for every }i.
\]
This notion does not depend on the choice of representatives $z,w$ of classes. We define the genus number of $f:N\to M$ by $g_f=\#(H_1(N)/(\tau-1)H_1(N))$.
\begin{thm}\label{thm.genus.ldt}
Assume the above situation. Consider the group homomorphism $\chi:H_1(N)\to \prod_{i=1}^r\mathbb{Z}/e_i\mathbb{Z}; [z]\mapsto ({\rm lk}(f_*(z),K_i)\ {\rm mod}\ e_i)$. Then, 
\[
{\rm Im}\chi=\{(x_i)\in\prod_{i=1}^r\mathbb{Z}/e_i\mathbb{Z}\mid\sum_{i=1}^r \frac{n}{e_i}x_i\equiv 0\mbox{ mod }n\},\ {\rm Ker}\chi=(\tau-1)H_1(N).
\]
 and
 \[H_1(N)/\approx\ \  \simeq H_1(N)/(\tau-1)H_1(N) \simeq {\rm Im}\chi.
\]
In particular, $$g_f=\frac{\prod_i e_i}{n}.$$
\end{thm}
\begin{lem}\label{well-def}
The following homomorphism $\chi:H_1(N)\to \prod_{i=1}^r\mathbb{Z}/e_i\mathbb{Z}; [z]\mapsto ({\rm lk}(f_*(z),K_i)\ {\rm mod}\ e_i)$ is well-defined.
\end{lem}
\begin{proof}[Proof of \cref{well-def}]
Put $\widetilde{L_0}:=f^{-1}(L_0)$ and $Y:=N\setminus\widetilde{L_0}$. Take any $[z],[w]\in H_1(N)$ with $[z]=[w]$. We show that lk$(f_*(z),K_i)\equiv$ lk$(f_*(w),K_i)$ mod $e_i$ for every $i$ i.e. $\chi([z-w])=0$. Since the inclusion $ \iota_{Y,N}:Y\hookrightarrow N$ induces the surjective homomorphism $\iota_{Y,N*}:H_1(Y)\to H_1(N)$, there exists $[v]\in H_1(Y)$ such that $\iota_{Y,N*}([v])=[z-w]$. For the pair $(N,Y)$, we have the following relative homology exact sequence
\begin{gather*}
\cdots\to H_2(N,Y)\overset{\partial}{\to}H_1(Y)\to H_1(N)\to\cdots.
\end{gather*}
By $[z-w]=0$, $[v]\in{\rm Ker }\iota_{Y,N*}={\rm Im}\partial$. 

Here, we have the excision isomorphism $H_2(N,Y)\cong H_2(V_{\widetilde{L_0}},\partial V_{\widetilde{L_0}})$ and the following relative homology exact sequence
\begin{gather*}
0\to H_2(V_{\widetilde{L_0}},\partial V_{\widetilde{L_0}}) \overset{\partial}{\to}H_1(\partial V_{\widetilde{L_0}})\to H_1(V_{\widetilde{L_0}}) \to 0.
\end{gather*}
Thus, $H_2(N,Y)\cong H_2(V_{\widetilde{L_0}},\partial V_{\widetilde{L_0}})\cong \langle[S_{J_j}]\mid S_{J_j}$ is a surface such that $\partial S_{J_j}=\mu_{J_j}$ with $J_j\subset\widetilde{L_0}\rangle$. 

Since $[v]\in{\rm Im}\partial$, there exist some $\sum_jc_{J_j}[S_{J_j}]\in H_2(V_{\widetilde{L_0}},\partial V_{\widetilde{L_0}})$ with $c_{J_j}\in\mathbb{Z}$ such that $[v]=\partial(\sum_jc_{J_j}[S_{J_j}])=\sum_jc_{J_j}[\mu_{J_j}]\in H_1(Y)$. Thus, $f_*([v])=\sum_{K\subset L_0}e_K(\sum_{J_j\subset f^{-1}(K)}c_{J_j})[\mu_{K}]$, where $e_K$ is a branch index of $K$. For any $i,i'\in\{1,\cdots,r\}$, if $i\neq i'$, then lk$([\mu_{K_{i'}}],K_i)=0$ and if $i=i'$, then lk$(e_i(\sum c_{J_j})[\mu_{K_i}],K_i)\equiv0$ mod $e_i$. Therefore, $$\chi([z-w])=({\rm lk}(\sum_{K\subset L_0}e_K(\sum_{J_j\subset f^{-1}(K)}c_{J_j})[\mu_K],K_i)\mbox{ mod }e_i)=0.\qedhere$$ \end{proof}

\begin{proof}[Proof of \cref{thm.genus.ldt}] 
  \textbf{Step 0.} By \cref{exVAL}, we can take a very admissible link $\mathcal{L}$ containing $L_0$ of $M$. By definition, the preimage  $\widetilde{\mathcal{L}}:=f^{-1}(\mathcal{L})$ is again a very admissible link of $N$. For the pairs $(N,\widetilde{\mathcal{L}}),(M,\mathcal{L})$, we have the id\`ele groups $I_{N,\widetilde{\mathcal{L}}},I_{M,\mathcal{L}}$, the diagonal maps $\Delta_{N,\widetilde{\mathcal{L}}},\Delta _{M,\mathcal{L}}$, the principal id\`ele groups $P_{N,\widetilde{\mathcal{L}}},P_{M,\mathcal{L}}$, and the unit id\`ele groups $U_{N,\widetilde{\mathcal{L}}},U_{M,\mathcal{L}}$, respectively.  

\textbf{Step 1.} Note that the natural homomorphism $f_*:I_{N,\widetilde{\mathcal{L}}}\to I_{M,\mathcal{L}}$ induces a natural surjective homomorphism 
$$\begin{array}{cccc}
\bar f:&I_{N,\widetilde{\mathcal{L}}}/P_{N,\widetilde{\mathcal{L}}}& \to &(f_*(I_{N,\widetilde{\mathcal{L}}})+P_{M,\mathcal{L}})/P_{M,\mathcal{L}}\\ &a+P_{N,\widetilde{\mathcal{L}}}&\mapsto &f_*(a)+P_{M,\mathcal{L}}.
\end{array}$$
We will show that Ker$\bar f=(\tau-1)(I_{N,\widetilde{\mathcal{L}}}/P_{N,\widetilde{\mathcal{L}}})$.  Take any $a+P_{N,\widetilde{\mathcal{L}}}\in$ Ker$\bar f$. By $\bar f(a+P_{N,\widetilde{\mathcal{L}}})=f_*(a)+P_{M,\mathcal{L}}=P_{M,\mathcal{L}}$, we have $f_*(a)\in P_{M,\mathcal{L}}$. By \cref{HNP}, there exists $A\in H_2(N,\widetilde{\mathcal{L}})$ such that $f_*(a)=f_*(\Delta_{N,\widetilde{\mathcal{L}}}(A))$, i.e., $f_*(a-\Delta_{N,\widetilde{\mathcal{L}}}(A))=0$. Moreover, by \cref{Hil90.ldt}, there exists $b\in I_{N,\widetilde{\mathcal{L}}}$ such that $a-\Delta_{N,\widetilde{\mathcal{L}}}(A)=(\tau-1)b$. Since $\Delta_{N,\widetilde{\mathcal{L}}}(A)\in P_{N,\widetilde{\mathcal{L}}}$, we have $$a+P_{N,\widetilde{\mathcal{L}}}=(\tau-1)b+P_{N,\widetilde{\mathcal{L}}}\in(\tau-1)(I_{N,\widetilde{\mathcal{L}}}/P_{N,\widetilde{\mathcal{L}}}).$$ Also, take any $(\tau-1)a+P_{N,\widetilde{\mathcal{L}}}\in (\tau-1)(I_{N,\widetilde{\mathcal{L}}}/P_{N,\widetilde{\mathcal{L}}})$. Since $\tau$ satisfies $f\circ\tau=f$, we have $\bar f((\tau-1)a+P_{N,\widetilde{\mathcal{L}}})=f_*((\tau-1)a)+P_{M,\mathcal{L}}=P_{M,\mathcal{L}}$. Thus, we have $(\tau-1)a+P_{N,\widetilde{\mathcal{L}}}\in$ Ker$\bar f$, and hence  Ker$\bar f=(\tau-1)(I_{N,\widetilde{\mathcal{L}}}/P_{N,\widetilde{\mathcal{L}}})$ holds as desired. 

By \cref{IdeleHom}, we have an isomorphism $\rho:I_{N,\widetilde{\mathcal{L}}}/(P_{N,\widetilde{\mathcal{L}}}+U_{N,\widetilde{\mathcal{L}}})\overset{\cong}{\to} H_1(N)$. Thus, the following commutative exact diagram commutes.
\[
\begin{array}{ccccccccc}
&&&&0&&0&&\\
&&&&\downarrow&&\downarrow&&\\
&&&&(U_{N,\widetilde{\mathcal{L}}}+P_{N,\widetilde{\mathcal{L}}})/P_{N,\widetilde{\mathcal{L}}}&\overset{\bar f}{\to}&(f_*(U_{N,\widetilde{\mathcal{L}}})+P_{M,\mathcal{L}})/P_{M,\mathcal{L}}&\to&0\\
&&&&\downarrow&&\downarrow&&\\
0&\to&(\tau-1)(I_{N,\widetilde{\mathcal{L}}}/P_{N,\widetilde{\mathcal{L}}})&\to&I_{N,\widetilde{\mathcal{L}}}/P_{N,\widetilde{\mathcal{L}}}&\overset{\bar f}{\to}&(f_*(I_{N,\widetilde{\mathcal{L}}})+P_{M,\mathcal{L}})/P_{M,\mathcal{L}}&\to&0\\
&&\downarrow&&\rho\downarrow&&&&\\
0&\to&(\tau-1)H_1(N)&\to&H_1(N)&&&&\\
&&\downarrow&&\downarrow&&&&\\
&&0&&0&&&&
\end{array}
\]
By this diagram, we obtain the following exact sequence $$0\to (\tau-1)H_1(N)\to H_1(N)\overset{f_*}{\to}  (f_*(I_{N,\widetilde{\mathcal{L}}})+P_{M,\mathcal{L}})/(f_*(U_{N,\widetilde{\mathcal{L}}})+P_{M,\mathcal{L}})\to 0.$$

\textbf{Step 2.} Next, we will show that $$(f_*(I_{N,\widetilde{\mathcal{L}}})+P_{M,\mathcal{L}})/(f_*(U_{N,\widetilde{\mathcal{L}}})+P_{M,\mathcal{L}})\cong \{(x_i)\in\prod_{i=1}^r\mathbb{Z}/e_i\mathbb{Z}\mid\sum_{i=1}^r \frac{n}{e_i}x_i\equiv 0\mbox{ mod }n\}.$$

By \cref{IdeleHom}, we have $I_{M,\mathcal{L}}=U_{M,\mathcal{L}}\oplus P_{M,\mathcal{L}}$. Hence, for any class of $I_{M,\mathcal{L}}/P_{M,\mathcal{L}}\cong U_{M,\mathcal{L}}$, there exists a unique representative $\alpha=(\alpha_K)_K$ such that $\alpha_K=r_K[\mu_K]$ with $r_K\in\mathbb{Z}$. We define a surjective group morphism by 
$$\begin{array}{cccc}
\psi:&I_{M,\mathcal{L}}/P_{M,\mathcal{L}}&\to&\prod_i\mathbb{Z}/e_i\mathbb{Z}\\ &(\alpha_K)_K+P_{M,\mathcal{L}}&\mapsto&(r_{K_i}{\rm lk}([\mu_{K_i}],K_{i}){\mbox{ mod }e_i})_i=(r_{K_i}\mbox{ mod }e_i)_i.\end{array}$$
We will show that Ker$\psi=(f_*(U_{N,\widetilde{\mathcal{L}}})+P_{M,\mathcal{L}})/P_{M,\mathcal{L}}$. To prove Ker$\psi\subset(f_*(U_{N,\widetilde{\mathcal{L}}})+P_{M,\mathcal{L}})/P_{M,\mathcal{L}}$,  take any $(r_K[\mu_K])_K+P_{M,\mathcal{L}}\in$ Ker$\psi$. Since $\psi((r_K[\mu_K])_K+P_{M,\mathcal{L}})=(r_{K_i} \mbox{ mod } e_i)_i=0$, there exists $s_{K_i}\in\mathbb{Z}$ such that $r_{K_i}=e_{K_i}s_{K_i}$ for each $K_i\subset L_0$. In addition, if $K\subset \mathcal{L}\setminus L_0$, then we have $e_K=1$, so we put $r_K=s_K$. For each $K\subset\mathcal{L}$, write $f^{-1}(K)=\sqcup_j J_j=J_1\sqcup\cdots\sqcup J_{c_K}$, and put $u_{J_j}:=\begin{cases}s_K[\mu_{J_1}]\mbox{ (if }j=1)\\ 0\mbox{ (otherwise)}\end{cases}$. Then, we have $\sum_{J_j\subset f^{-1}(K)}{f_{\partial V_{J_j}}}_*(u_{J_j})=r_{K}[\mu_{K}]$ for each $K$. Thus, we obtain $u=(u_J)_J\in U_{N,\widetilde{\mathcal{L}}}$ satisfying $(r_K[\mu_K])_K=f_*(u)$. This implies that $(r_K[\mu_K])_K+P_{M,\mathcal{L}}\in (f_*(U_{N,\widetilde{\mathcal{L}}})+P_{M,\mathcal{L}})/P_{M,\mathcal{L}}$. 

Also, to prove Ker$\psi\supset(f_*(U_{N,\widetilde{\mathcal{L}}})+P_{M,\mathcal{L}})/P_{M,\mathcal{L}}$, we take any $\alpha+P_{M,\mathcal{L}}\in (f_*(U_{N,\widetilde{\mathcal{L}}})+P_{M,\mathcal{L}})/P_{M,\mathcal{L}}$. There exists $(u_J)_J\in U_{N,\widetilde{\mathcal{L}}}$ such that $\alpha+P_{M,\mathcal{L}}=f_*((u_J)_J)+P_{M,\mathcal{L}}$. It suffices to look at $\alpha_K$ only for the case with $K=K_i\subset L_0$. If $J\subset f^{-1}(K_i)$, then ${f_{\partial V_{J}}}_*(u_J)=e_{K_i}m_J[\mu_{K_i}]$. Thus, $\psi(f_*((u_J)_J)+P_{M,\mathcal{L}})=0$. Hence, Ker$\psi=(f_*(U_{N,\widetilde{\mathcal{L}}})+P_{M,\mathcal{L}})/P_{M,\mathcal{L}}$ is proved. 

Furthermore, by \cref{GCFT}, there exist isomorphisms $C_{M,\mathcal{L}}/f_*(C_{N,\widetilde{\mathcal{L}}})\overset{\cong}{\to}I_{M,\mathcal{L}}/(f_*(I_{N,\widetilde{\mathcal{L}}})+P_{M,\mathcal{L}})\overset{\cong}{\to} {\rm Gal}(f)$. Thus, we obtain the following commutative exact diagram
\[
\begin{array}{ccccccccc}
&&0&&0&&&&\\
&&\downarrow&&\downarrow&&&&\\
&&(f_*(U_{N,\widetilde{\mathcal{L}}})+P_{M,\mathcal{L}})/P_{M,\mathcal{L}}&\overset{{\rm id}}{\to}&(f_*(U_{N,\widetilde{\mathcal{L}}})+P_{M,\mathcal{L}})/P_{M,\mathcal{L}}&&&&\\
&&\downarrow&&\downarrow&&&&\\
0&\to&(f_*(I_{N,\widetilde{\mathcal{L}}})+P_{M,\mathcal{L}})/P_{M,\mathcal{L}}&\to&I_{M,\mathcal{L}}/P_{M,\mathcal{L}}&\to&{\rm Gal}(f)&\to&0\\
&&&&\psi\downarrow&&\wr\downarrow&&\\
&&&&\prod_i\mathbb{Z}/e_i\mathbb{Z}&\overset{\Sigma}{\to}&\mathbb{Z}/n\mathbb{Z}&&\\
&&&&\downarrow&&\downarrow&&\\
&&&&0&&0&&.
\end{array}
\]
Here, we define a morphism $\Sigma:\prod_i\mathbb{Z}/e_i\mathbb{Z}\to\mathbb{Z}/n\mathbb{Z};(x_i)_i\mapsto\sum_{i=1}^r\frac{n}{e_i}x_i$. This diagram induces a short exact sequence $$0\to f_*(U_{N,\widetilde{\mathcal{L}}})+P_{M,\mathcal{L}}\to f_*(I_{N,\widetilde{\mathcal{L}}})+P_{M,\mathcal{L}}\overset{\bar{\psi}}{\to}{\rm Ker}\Sigma\to0.$$

\textbf{Step 3.} Finally, we will show that $\chi=(\bar{\psi}\circ f_*)\circ\rho^{-1}$. For this purpose, we put $\bar{\chi}:=\chi\circ\rho$, and prove that the following diagram commutes: 
$$
\xymatrix{
I_{N,\widetilde{\mathcal{L}}}/({P_{N,\widetilde{\mathcal{L}}}+U_{N,\widetilde{\mathcal{L}}}}) \ar@{->}[r]^-{f_*} \ar@{->}[d]^{\cong}_{\rho} \ar@{->}[dr]^{\bar{\chi}}&(f_*(I_{N,\widetilde{\mathcal{L}}})+ P_{M,\mathcal{L}})/({f_*(U_{N,\widetilde{\mathcal{L}}})+P_{M,\mathcal{L}}}) \ar@{^{(}->}[r] &I_{M,\mathcal{L}}/({f_*(U_{N,\widetilde{\mathcal{L}}})+P_{M,\mathcal{L}}}) \ar@{->}[dl]^{\bar{\psi}}\\
H_{1}(N) \ar@{->}[r]^{\chi}& \prod_{i=1}^{r}\mathbb{Z}/e_i\mathbb{Z}&.
}$$  
For every knot $J\subset\widetilde{\mathcal{L}}$, the inclusion $\iota_J:\partial V_J\hookrightarrow N$ induces a homomorphism $\iota_{J,*}:H_1(\partial V_J)\to H_1(N)$. We take any $[a]=[(a_J)_J]\in I_{N,\widetilde{\mathcal{L}}}/(U_{N,\widetilde{\mathcal{L}}}+P_{N,\widetilde{\mathcal{L}}})$. By the definition of an id\`ele, there exists a finite link $\Lambda\subset\widetilde{\mathcal{L}}$ such that for any $J\not\subset \Lambda$, $a_J=m_J[\mu_J]$. We have $\rho([a])=\sum_{J\subset \mathcal{L}}\iota_{J,*}(a_J)=\sum_{J\subset \Lambda}l_J[\lambda_J]$. Noting that lk$(\lambda_K,K)=0$ for each $K\subset L_0$, $$\bar\chi([a])=({\rm lk}(f_*(\sum_{J\subset \Lambda}l_J[\lambda_J]),K_i)\mbox{ mod }e_i)_i.$$
On the other hand, we consider $\bar{\psi}\circ f_*([a])$. By $u':=(m_J[\mu_J])_J\in U_{N,\widetilde{\mathcal{L}}}$, we have $[a]=[a-u']\in I_{N,\widetilde{\mathcal{L}}}/(U_{N,\widetilde{\mathcal{L}}}+P_{N,\widetilde{\mathcal{L}}})$. Define  $\beta=(\beta_K)_K\in I_{M,\mathcal{L}}$ by $\beta_K={\rm lk}(f_*(\sum_{J\subset\Lambda}l_J[\lambda_J]),K)[\mu_K]$. By \cref{diagonal}, $f_*(a-u')-\beta\in P_{M,\mathcal{L}}$ implies that 
$$f_*([a])=f_*([a-u'])=[\beta]=[({\rm lk}(f_*(\sum_{J\subset\Lambda}l_J[\lambda_J]),K)[\mu_K])_K]\in\frac{f_*(I_{N,\widetilde{\mathcal{L}}})+ P_{M,\mathcal{L}}}{f_*(U_{N,\widetilde{\mathcal{L}}})+P_{M,\mathcal{L}}}.$$
Thus, we obtain $$(\bar{\psi}\circ f_*)([a])=({\rm lk}(f_*(\sum_{J\subset\Lambda}l_J[\lambda_J]),K_i)\mbox{ mod }e_i)_i=\bar{\chi}([a])$$
and hence the above diagram commutes. Therefore, by $\chi=\bar{\chi}\circ\rho^{-1}$, we obtain the assertion of the theorem as desired. 
\end{proof}

\noindent 
\textbf{Acknowledgement}.  I would like to thank my supervisor Professor Masanori Morishita for proposing the problem to find a topological analogue for 3-manifolds of Hilbert Satz 90 and the genus theory, and for his advice and encouragement. The author is very grateful to Jun Ueki for his advice and helpful comments.

\bibliographystyle{plainurl}
\bibliography
{Tashiro.refs.bib}

\begin{thebibliography}{10}

\bibitem{CasselsFrohlich}
J.W.S. Cassels and A.~Fr{\"o}hlich.
\newblock {\em Algebraic number theory, Proceedings of an Instructional
  Conference}.
\newblock Academic Press, 1967.

\bibitem{Furuta1967}
Yoshiomi Furuta.
\newblock The genus field and genus number in algebraic number fields.
\newblock {\em Nagoya Math. J.}, 29:281--285, 1967.

\bibitem{IyanagaTamagawa}
S.~Iyanaga and T.~Tamagawa.
\newblock Sur la theorie du corps de classes sur le corps des nombres
  rationnels.
\newblock {\em Journal of the Mathematical Society of Japan}, 3(1):220--227,
  1951.
\newblock \href {https://doi.org/10.2969/jmsj/00310220}
  {\path{doi:10.2969/jmsj/00310220}}.

\bibitem{Mihara2019Canada}
Tomoki Mihara.
\newblock Cohomological approach to class field theory in arithmetic topology.
\newblock {\em Canad. J. Math.}, 71(4):891--935, 2019.
\newblock \href {https://doi.org/10.4153/cjm-2018-020-0}
  {\path{doi:10.4153/cjm-2018-020-0}}.

\bibitem{Morishita2001g}
Masanori Morishita.
\newblock A theory of genera for cyclic coverings of links.
\newblock {\em Proc. Japan Acad. Ser. A Math. Sci.}, 77(7):115--118, 2001.
\newblock URL: \url{http://projecteuclid.org/euclid.pja/1148393034}.

\bibitem{Morishita2024}
Masanori Morishita.
\newblock {\em Knots and primes -An introduction to Arithmetic Topology}.
\newblock Universitext. Springer Singapore, 2024.
\newblock \href {https://doi.org/10.1007/978-981-99-9255-3}
  {\path{doi:10.1007/978-981-99-9255-3}}.

\bibitem{Neukirch}
J{\"u}rgen Neukirch.
\newblock {\em Algebraic number theory}, volume 322 of {\em Grundlehren der
  Mathematischen Wissenschaften [Fundamental Principles of Mathematical
  Sciences]}.
\newblock Springer-Verlag, Berlin, 1999.
\newblock Translated from the 1992 German original and with a note by Norbert
  Schappacher, With a foreword by G. Harder.
\newblock \href {https://doi.org/10.1007/978-3-662-03983-0}
  {\path{doi:10.1007/978-3-662-03983-0}}.

\bibitem{Niibo1}
Hirofumi Niibo.
\newblock Id\`elic class field theory for 3-manifolds.
\newblock {\em Kyushu J. Math.}, 68(2):421--436, 2014.
\newblock \href {https://doi.org/10.2206/kyushujm.68.421}
  {\path{doi:10.2206/kyushujm.68.421}}.

\bibitem{NiiboUeki}
Hirofumi Niibo and Jun Ueki.
\newblock Id\`elic class field theory for 3-manifolds and very admissible
  links.
\newblock {\em Trans. Amer. Math. Soc.}, 371(12):8467--8488, 2019.
\newblock \href {https://doi.org/10.1090/tran/7480}
  {\path{doi:10.1090/tran/7480}}.

\bibitem{NiiboUeki2023RMS}
Hirofumi Niibo and Jun Ueki.
\newblock A {H}ilbert reciprocity law on 3-manifolds.
\newblock {\em Res. Math. Sci.}, 10(1):Paper No. 3, 8, 2023.
\newblock \href {https://doi.org/10.1007/s40687-022-00364-w}
  {\path{doi:10.1007/s40687-022-00364-w}}.

\bibitem{Tashiro2024}
Hirotaka Tashiro.
\newblock On {H}asse norm principle for 3-manifolds in arithmetic topology.
\newblock {\em Res. Math. Sci.}, 12(2):Paper No. 27, 10, 2025.
\newblock \href {https://doi.org/10.1007/s40687-025-00508-8}
  {\path{doi:10.1007/s40687-025-00508-8}}.

\bibitem{Ueki1}
Jun Ueki.
\newblock On the homology of branched coverings of 3-manifolds.
\newblock {\em Nagoya Math. J.}, 213:21--39, 2014.

\bibitem{Ueki3}
Jun Ueki.
\newblock On the {I}wasawa {$\mu$}-invariants of branched
  {$\bold{Z}_p$}-covers.
\newblock {\em Proc. Japan Acad. Ser. A Math. Sci.}, 92(6):67--72, 2016.
\newblock URL: \url{https://doi.org/10.3792/pjaa.92.67}.

\bibitem{Yokoi1967}
Hideo Yokoi.
\newblock On the class number of a relatively cyclic number field.
\newblock {\em Nagoya Math. J.}, 29:31--44, 1967.

\end{thebibliography}
\


\leavevmode\\
Hirotaka Tashiro\\
Faculty of Mathematics, Kyushu University\\
744, Motooka, Nishi-ku, Fukuoka, 819-0395, JAPAN\\
e-mail: tashiro.hirotaka.035@s.kyushu-u.ac.jp

\end{document}